\documentclass{amsart}

\usepackage{amssymb}
\usepackage{amsmath}
\usepackage[all]{xy}
\usepackage[english]{babel}
\usepackage{tikz}
\usetikzlibrary{matrix}
\usetikzlibrary{calc,intersections}
\usepackage[shortlabels]{enumitem}
\usepackage{amssymb}
\usepackage{faktor}
\usepackage{nicefrac}
\usepackage[ruled]{algorithm2e}
\usepackage[top=2.5cm,left=2.5cm,right=2.5cm]{geometry} 
 
\usepackage[font=small,labelfont=bf]{caption} 

\usepackage{hyperref} 

\DeclareMathOperator{\Pic}{Pic}
\DeclareMathOperator{\Hom}{Hom}
\DeclareMathOperator{\GL}{GL}

\DeclareMathOperator{\Tr}{Tr}
\newcommand{\bICM}[1]{\ICM_{#1}}
\DeclareMathOperator{\ICM}{ICM}
\DeclareMathOperator{\Wk}{\mathcal W}
\newcommand{\bWk}[1]{\Wk_{#1}}
\def\Q{\mathbb{Q}}
\def\Z{\mathbb{Z}}

\newcommand{\cA}{{\mathcal A}}
\newcommand{\cB}{{\mathcal B}}
\newcommand{\cC}{{\mathcal C}}
\newcommand{\cI}{{\mathcal I}}
\newcommand{\cL}{{\mathcal L}}
\newcommand{\cM}{{\mathcal M}}
\newcommand{\cO}{{\mathcal O}}
\newcommand{\cP}{{\mathcal P}}
\newcommand{\frf}{{\mathfrak f}}
\renewcommand{\frm}{{\mathfrak m}}
\newcommand{\p}{{\mathfrak p}}
\newcommand{\frP}{{\mathfrak P}}

\renewcommand{\bar}{\overline}
\newcommand{\vphi}{{\varphi}}
\newcommand{\isoclass}[1]{\{ #1 \}}
\newcommand{\wkclass}[1]{[ #1 ]}
\newcommand{\set}[1]{\left\lbrace#1\right\rbrace }
\newcommand{\Span}[1]{\left<#1\right>}

\newtheorem{thm}{Theorem}[section] 
\newtheorem{lemma}[thm]{Lemma}     
\newtheorem{cor}[thm]{Corollary}
\newtheorem{prop}[thm]{Proposition}
\newtheorem{df}[thm]{Definition}
\newtheorem{remark}[thm]{Remark}
\newtheorem{example}[thm]{Example}

\author{Stefano Marseglia}
\address{Matematiska institutionen, Stockholms universitet, Sweden}
\curraddr{Mathematical Institute, Utrecht University, The Netherlands}
\email{s.marseglia@uu.nl}
\title{Computing the ideal class monoid of an order}


\begin{document}
\maketitle

\let\thefootnote\relax\footnote{
This is the accepted version of the following article:
\emph{Marseglia, Stefano;
Computing the ideal class monoid of an order.
J. Lond. Math. Soc. (2) 101 (2020), no. 3, 984-1007}, which has been published in final form at
\url{http://dx.doi.org/10.1112/jlms.12294}
}
\vspace{-1cm}
\begin{abstract}
There are well known algorithms to compute the class group of the maximal order $\cO_K$ of a number field $K$ and the group of invertible ideal classes of a non-maximal order $R$.
In this paper we explain how to compute also the isomorphism classes of non-invertible ideals of an order $R$ in a finite product of number fields $K$. In particular we also extend the above-mentioned algorithms to this more general setting.
As an application, we use a generalization of a theorem of Latimer and MacDufee to produce algorithms that return representatives of all conjugacy classes of integral matrices with given characteristic polynomial (satisfying certain assumptions) and solve the conjugacy problem for such matrices.
\end{abstract}


\section{Introduction}


Let $K$ be a number field and $R$ an order in $K$.
There are well known algorithms to compute the ideal class group $\Pic(R)$ when $R$ is the ring of integers $\cO_K$ of $K$, also known as the maximal order, see for example \cite{cohen93}.
This information can be used to efficiently compute the group $\Pic(R)$ of invertible ideal classes of a non-maximal order $R$,
as is explained in \cite{klupau05}.

On the other hand not much is known about non-invertible ideals and, in particular, it is not known how to compute the monoid of all ideal classes of $R$, which we will denote $\ICM(R)$.
In the literature one can find results about the local isomorphism classes of ideals.
More precisely, one studies the genus of an ideal, which is its isomorphism class after localizing at a rational prime $p$, or its weak equivalence class, which is its isomorphism class after localizing at a prime ideal $\p$ of $R$.
For the notion of genus we refer to  \cite{Reiner70} and \cite{Reiner03}, while for results about the weak equivalence classes we cite \cite{dadetz62}.
It is important to mention that these two apparently different notions are actually equivalent, as pointed out in \cite[Section 5]{LevyWiegand85}.

In the present paper we exhibit:
\begin{itemize}
    \item an algorithm to compute the monoid $\ICM(R)$ of isomorphism classes of fractional ideals of an order $R$ in a finite product of number fields $K$, see Theorem \ref{thm:icmcomputation}, Proposition \ref{prop:wkbar} and the algorithms in Section \ref{sec:algorithms}, and
    \item a bijection between the set of conjugacy classes of integral matrices with given square-free minimal polynomial $m$ and characteristic polynomial $c$ and the set of $R$-isomorphism classes of $\Z$-lattices in a certain $\Q$-algebra, where $R$ is an order in a certain product of number fields, see Theorem \ref{thm:matrixconj}.
    Under certain assumptions on the polynomials $c$ and $m$, we can reduce such a description to an ideal class monoid computation and hence solve the conjugacy problem and produce representatives of each conjugacy class, see Corollary \ref{cor:computeconjclass}.
\end{itemize}
Theorem \ref{thm:matrixconj} is a generalization of the main result of \cite{LaClMD33} where the case when $c$ is square-free is analyzed.
Their theorem was then reproved with a different method under the extra assumption that $c$ is irreducible in \cite{taussky49}.
The author recently discovered that Theorem \ref{thm:matrixconj} has independently been proved in \cite{HusertPhDThesis17} in greater generality.
Moreover, a new algorithm to test whether two rational matrices are conjugate over $\Z$ is given in \cite{EickHofmannOBrien19}.
This algorithm works in more generality than ours at the cost of being slower.

The present paper is structured as follows. 
In Section \ref{sec:orders} we recall the definitions of an order $R$ and a fractional ideal in a product of number fields $K$ and some basic results, which are well known in the case that $K$ is a number field.
In Section \ref{sec:idealclasses} we introduce isomorphisms of fractional ideals, and the monoid that the corresponding classes form, called the ideal class monoid $\ICM(R)$.
Since it is hard to compute $\ICM(R)$ directly, in Section \ref{sec:wkclasses} we relax the notion of isomorphism to a local one, called weak equivalence.
We explain how to effectively check whether two fractional ideals are weakly equivalent and how to algorithmically reconstruct $\ICM(R)$ once we have computed $\Pic(S)$, for every over-order $S$ of $R$, and the monoid of weak equivalence classes $\Wk(R)$.
In Section \ref{sec:computingwkclasses} we explain a concrete way to compute representatives of the weak equivalence class monoid $\Wk(R)$.
In Section \ref{sec:algorithms} we give the pseudo-code of the algorithms described in the previous sections and discuss the running time and the bottlenecks.
In Section \ref{sec:examples} we present some concrete calculations of ideal class monoids.
Finally, in Section \ref{sec:matrixconj} we present our results about computing conjugacy classes of integral matrices and compare the running times of our algorithm and the algorithm proposed in \cite{EickHofmannOBrien19}, see Example \ref{ex:timings}.
Observe that Corollary \ref{cor:computeconjclass} gives a solution to \cite[Problem 7.7]{EickHofmannOBrien19} for matrices satisfying the appropriate hypotheses. 

The algorithms have been implemented in Magma \cite{Magma} and the code is available at \url{https://github.com/stmar89/AbVarFq}.

Another application, namely computing isomorphism classes of abelian varieties defined over a finite field belonging to an isogeny class determined by a square-free Weil polynomial, is discussed in \cite{MarAbVar18}.

\subsection*{Acknowledgments}
The author would like to thank Jonas Bergstr\"om for helpful discussions and Rachel Newton, Christophe Ritzenthaler, Peter Stevenhagen and Marco Streng for comments on a previous version of the paper, which is part of the author's Ph.D thesis \cite{MarPhDThesis18}.
The author wishes to thank J\"urgen Kl\"uners for pointing out the reference \cite{HusertPhDThesis17}.
We thank Bettina Eick, Tommy Hofmann and Eamonn O'Brien for allowing us to read an early version of their article \cite{EickHofmannOBrien19}.
We are grateful to the anonymous referee of the \textit{Journal of London Mathematical Society} for reading the paper carefully and providing thoughtful comments.


\section{Orders}
\label{sec:orders}


In what follows, the word ring will mean commutative ring with unit.
An \emph{order} is a reduced ring $R$, which is free and finitely generated as a $\Z$-module. Let $K$ be the total quotient ring of an order $R$, that is, the localization of $R$ at the multiplicative set of non-zero-divisors.
Then $K$ is an \'etale algebra over $\Q$ with $R\otimes_\Z \Q = K$, and in particular $K$ is a finite product of number fields, say $K= K_1\times \ldots \times K_r$. 
The set of orders in $K$ contains a maximal element with respect to the inclusion relation.
This order, denoted $\cO_K$, is the integral closure of $\Z$ in $K$ and it is usually referred to as the \emph{maximal order} or the \emph{ring of integers} of $K$.
Note that $\cO_K=\cO_{K_1} \times \ldots \cO_{K_r}$, where $\cO_{K_i}$ is the maximal order of $K_i$.
Indeed, $\cO_K$ contains $\prod_i\cO_{K_i}$, so $\cO_K$ is a product of orders $S_i$ in $K_i$ and by maximality it follows that $S_i=\cO_{K_i}$. 
There are well known algorithms to compute each $\cO_{K_i}$, see for example \cite[Chapter 6]{cohen93}, and in what follows we will assume that we can compute $\cO_K$.

From now on $R$ will be an order in $K$.
A finitely generated sub-$R$-module $I$ of $K$ is called a \emph{fractional $R$-ideal} if $I\otimes_\Z \Q=K$.
Such an $I$ is a finitely generated free $\Z$-module of the same rank as $R$, and so we can find $x_1,\ldots, x_n \in K$, where $n=\sum_{i=1}^r[K_i:\Q]$, such that
\[I=x_1\Z\oplus\ldots\oplus x_n\Z.\]
In particular, if $I\subseteq R$ then the quotient $R/I$ is finite.
We denote by $\cI(R)$ the set of all fractional ideals of $R$. 
Observe that for every fractional $R$-ideal $I$, there exists a non-zero-divisor $x\in K$ such that $xI$ is an ideal of $R$.
Moreover, every ideal of $R$ containing a non-zero-divisor is a fractional $R$-ideal.
The fractional $R$-ideals that are rings are called \emph{over-orders} of $R$. Since $\cO_K$ and $R$ have the same rank as free abelian groups, the quotient $\cO_K/R$ is finite and thus there are only finitely many over-orders of $R$.

Given two fractional $R$-ideals $I$ and $J$, the product $IJ$, the sum $I+J$, the intersection $I\cap J$, and the ideal quotient
\[(I:J)=\set{x\in K: xJ\subseteq I}\]
are fractional $R$-ideals.
In particular, ideal multiplication induces on $\cI(R)$ the structure of a commutative monoid with unit element $R$.
A useful property of the ideal quotient is the following lemma, whose proof we leave to the reader.
\begin{lemma}
\label{lemma:quotidealprod}
Let $I,J,L$ be fractional $R$-ideals, then
\[ ((I:J):L) = (I:JL).\]
\end{lemma}


If $I$ is a fractional $R$-ideal then $(I:I)$ is a sub-ring of $K$ containing $R$. Hence it is an over-order of $R$ and, in particular, it is the biggest over-order of $R$ for which $I$ is a fractional ideal. It is called the \emph{multiplicator ring} of $I$.

\begin{lemma}
\label{lemma:idempfracid}
 The over-orders of $R$ are precisely the idempotents of $\cI(R)$, that is, the  fractional $R$-ideals $S$ such that $SS=S$.
\end{lemma}
\begin{proof}
 Let $S$ be an over-order of $R$. Then $S$ is multiplicatively closed and contains $1$, so $SS=S$.
 Conversely, let $S$ be an idempotent fractional ideal of $R$. Let $T=(S:S)$ be the multiplicator ring of $S$. As $SS=S$ we have $S\subseteq T$ and hence $S$ is a finitely generated idempotent $T$-ideal.
 By the determinant trick  it must be generated by an idempotent element $e$ of $T$.
 As $S$ has full rank over $\Z$ we must have $e=1$, that is, $S=T$.
 In particular, $S$ is an over-order of $R$.
\end{proof}

We will denote by $\Tr_{K/\Q}$, or simply $\Tr$ when no confusion can arise, the \emph{trace form} on $K$, which associates to every $x\in K$ the trace of the matrix of the multiplication by $x$.
For every fractional $R$-ideal $I$, we define the \emph{trace dual ideal} as $I^t=\set{x\in K : \Tr(xI)\subseteq \Z}$.
Given a $\Z$-basis $\{x_i\}$ of $I$, we have $I^t=x_1^*\Z\oplus\ldots \oplus x_n^*\Z$, where $\{x_j^*\}$ is the \emph{trace dual basis}, which is characterized by $\Tr(x_ix_j^*)=1$ or $0$ according to whether $i=j$ or $i\neq j$.
Observe that $I^t$ is a fractional $R$-ideal and that the map $x\mapsto \vphi_x$, where $\vphi_x(y)=\Tr(xy)$ is an isomorphism from  $I^t$ to $\Hom_\Z(I,\Z)$.
%
%
In the next lemma we will summarize some well known properties of the trace dual ideal.
\begin{lemma}
 Let $R$ be an order in $K$, let $I$ and $J$ be two fractional $R$-ideals and let $x$ be in $K^\times$.
 Then the following holds:
 \begin{itemize}
  \item $(I^t)^t=I$,
  \item $I\subset J \Longleftrightarrow J^t\subset I^t$,
  \item $(I\cap J)^t = I^t + J^t$,
  \item $(x I)^t = \frac 1 x I^t$,
  \item $(I:J)=(I^t J)^t$,
  \item $(I:J)=(J^t:I^t)$,
  \item \label{tracedualmultiplicatorring} $S=(I:I) \Longleftrightarrow II^t=S^t$.
 \end{itemize}
\end{lemma}

Let $\p$ be a prime ideal of $R$ which is also a fractional $R$-ideal. Since the integral domain $R/\p$ is finite, we see that $\p$ is a maximal ideal.
Conversely, if $\frm$ is a maximal ideal of $R$ then it contains the prime $p$ which is the characteristic of the field $R/\frm$, and hence $\frm$ is a fractional $R$-ideal.
We will refer to the maximal ideals of $R$ as the \emph{primes} of $R$.
Since for any fractional $R$-ideal $I$ contained in $R$ the quotient $R/I$ is finite, we deduce that there exists only a finite number of primes of $R$ containing $I$.

A fractional $R$-ideal $I$ is said to be \emph{invertible} if $IJ=R$, for some fractional $R$-ideal $J$. Observe that if such a $J$ exists then $J=(R:I)$.

\begin{remark}
  Note that we could equivalently say that an $R$-ideal $I$ is invertible if and only if there exists an $R$-ideal $J$ and a non-zero-divisor $d$ such that $IJ=dR$.
  This characterization allows us to talk about invertible ideals in any ring.
\end{remark}

\begin{lemma}
\label{lemma:invidealmultring}
 Let $I$ be an invertible fractional $R$-ideal.
 Then $R$ is the multiplicator ring of $I$.
\end{lemma}
\begin{proof}
 Put $S=(I:I)$. Since $I$ is an $R$-module we have $R\subseteq S$ and using $I=SI$ we deduce that
 \[ R = I(R:I) = SI(R:I) = SR = S. \]
\end{proof}

\begin{remark}
  Let $I$ be a fractional $R$-ideal and let $S$ an over-order of $R$.
  In view of Lemma \ref{lemma:invidealmultring}, the expression ``$I$ is invertible as fractional $S$-ideal" implies that $S$ is the multiplicator ring of $I$.
\end{remark}

The following lemmas are useful for understanding how invertible ideals behave with respect to localizations at primes.

\begin{lemma}{\cite[Theorem 12.3]{Ka49}}
\label{lemma:PIRonmaximals}
 Let $T$ be a Noetherian ring. Then $T$ is a principal ideal ring if and only if every maximal ideal is principal.
\end{lemma}

\begin{lemma}{\cite[Proposition 7.4]{Gilmer92}}
\label{lemma:invertilbeprincipal}
 Let $T$ be a ring with finitely many maximal ideals and let $I$ be a $T$-ideal. Then $I$ is invertible in $T$ if and only if $I$ is principal and generated by a non-zero-divisor.
\end{lemma}
It is not difficult to prove the following.
\begin{lemma}
\label{lemma:invlocprinc}
 Let $I$ be a fractional $R$-ideal.
 Then $I$ is invertible if and only if $I_\p$ is principal for every prime $\p$ of $R$.
\end{lemma}
From the previous lemmas it is easy to deduce the following result.
\begin{cor}
\label{cor:pregular}
 Let $\p$ be a prime of $R$. Then $\p$ is invertible if and only if $R_\p$ is a principal ideal ring.
\end{cor}

Observe that $\cO_K$ is the only order in $K$ whose fractional ideals are all invertible.
We introduce now some classes of orders which are particularly well behaved in terms of invertibility of ideals.
\begin{prop}{\cite[Proposition 2.7]{buchlenstra}}
\label{prop:eqGorenstein}
 Let $R$ be an order with trace dual $R^t$. The following are equivalent:
 \begin{enumerate}[(a)]
  \item for every fractional $R$-ideal $I$, we have $(R:(R:I))=I$;
  \item for every fractional $R$-ideal $I$, we have $(I:I)=R$ if and only if $I$ is invertible;
  \item $R^t$ is invertible in $R$.
 \end{enumerate}
\end{prop}
An order satisfying one of the equivalent conditions of Proposition \ref{prop:eqGorenstein} is called \emph{Gorenstein}.
This definition is equivalent to the usual one, see \cite[Theorem 6.3]{basshy63}.
Observe that $\cO_K$ is Gorenstein, but there are Gorenstein orders which are not maximal. One class of examples of Gorenstein orders are the \emph{monogenic} orders, which are of the form $\Z[x]/(f)$, where $f$ is a monic polynomial with integer coefficients and non-zero discriminant.
\begin{cor}{\cite[Example 2.8]{buchlenstra}}
   Monogenic orders are Gorenstein.
\end{cor}

An order $R$ is called a \emph{Bass} order if every over-order of $R$ is Gorenstein, 
or equivalently, if the $R$-module $\cO_K/R$ is cyclic, that is, if $\cO_K=R+xR$ for some $x \in \cO_K$. For a proof and other equivalent characterizations, see for example \cite[Theorem 2.1]{LevyWiegand85} and Proposition \ref{prop:icmclifford}. Observe that every order in a quadratic number field is a Bass order.


\section{Ideal classes}
\label{sec:idealclasses}


Recall that for an order $R$ in $K$ we denote by $\cI(R)$ the commutative monoid of fractional $R$-ideals.
\begin{df}
 Let $R$ be an order in $K$. The \emph{ideal class monoid} of $R$ is
 \[\ICM(R)=\faktor{\cI(R)}{\simeq},\]
 where $I\simeq J$ if and only if $I$ and $J$ are isomorphic as $R$-modules. We will denote the ideal class of $I$ with $\isoclass{I}$.
\end{df}
The name is justified by the fact that $\ICM(R)$ inherits the commutative monoid structure of $\cI(R)$, as will become evident with Corollary \ref{cor:idemodprinc}.
\begin{lemma}
\label{lemma:isofracideals}
 Let $R$ be an order in $K$. Consider an $R$-module morphism $\vphi:I\to J$, where $I$ and $J$ are two fractional $R$-ideals, then $\vphi$ is a multiplication by some $\alpha \in K$.
\end{lemma}
The previous lemma, whose proof is left to the reader, directly implies the following two Corollaries.
\begin{cor}
\label{cor:isolocal}
  Let $I$ and $J$ be two fractional $R$-ideals. Then we have a natural identification
  \[ \Hom_R(I,J) = (J:I). \]
  In particular, if $\p$ is a prime of $R$ then every $R_\p$-linear morphism $\vphi:I_\p\to J_\p$ is a multiplication by some $\alpha$ in the total quotient ring of $R_\p$.
\end{cor}
\begin{cor}
\label{cor:idemodprinc}
  Two fractional $R$-ideals $I$ and $J$ are isomorphic if and only if there exists an $\alpha\in K^\times$ such that $I=\alpha J$.  
\end{cor}

The group $\cP(R)$ of principal fractional $R$-ideals acts by multiplication on $\cI(R)$ and we have that 
  \[\ICM(R)=\faktor{\cI(R)}{\cP(R)}.\]
Observe that every fractional ideal in $\cP(R)$ is invertible (as a fractional $R$-ideal), so we can consider the quotient of invertible fractional $R$-ideals by $\cP(R)$, which will inherit a group structure.
\begin{df}
 Let $R$ be an order in $K$. The \emph{Picard group} of $R$ is
 \[\Pic(R)=\faktor{\set{\text{invertible }I \in \cI(R)}}{\cP(R)}.\]
\end{df}

Since being invertible is a property of the ideal class, we can conclude that $\Pic(R)\subseteq \ICM(R)$. Observe that equality holds if and only if $R=\cO_K$.

Since $\cO_K$ is a finite product of Dedekind domains, we have that every ideal can be written in a unique way as a product of prime ideals, see for example \cite[Theorem 22.24]{Reiner03}.
For every invertible fractional ideals of non-maximal order $R$, we can find an isomorphic one, say $I$, which is coprime with the conductor $\frf=(R:\cO_K)$. This implies that $I\cO_K \cap R = I$ and hence it follows that $I$
admits a factorization into a product of primes of $R$.
But this is not true if we look at non-invertible ideals.

It is easy to show that the multiplicator ring is an invariant of the ideal class.
\begin{lemma}
 Let $R$ be an order in $K$. If two fractional $R$-ideals $I$ and $J$ are isomorphic then they have the same multiplicator ring.
\end{lemma}

It follows that
\begin{equation}
\label{eq:icmsuppics}
  \ICM(R) \supseteq \bigsqcup \Pic(S) 
\end{equation}
where the disjoint union is taken over the set of over-orders $S$ of $R$.

Recall that a commutative monoid is called \emph{Clifford} if it is a disjoint union of groups.
For other equivalent definitions of a commutative Clifford monoid see \cite[Section 1]{zz94} or \cite[Chapter IV]{Hel40}.
\begin{prop}
\label{prop:icmclifford}
 The following are equivalent:
 \begin{enumerate}[(a)]
  \item \label{prop:icmclifford:a} $R$ is Bass,
  \item \label{prop:icmclifford:b} the inclusion in \eqref{eq:icmsuppics} is an equality,
  \item \label{prop:icmclifford:c} $\ICM(R)$ is Clifford.
 \end{enumerate}
\end{prop}
\begin{proof}
 \ref{prop:icmclifford:a}$\Rightarrow$\ref{prop:icmclifford:b}:
 If $R$ is Bass then every over-order is Gorenstein and in particular every fractional $R$-ideal $I$ is invertible in its own multiplicator ring $S$.
 This means that $\isoclass{I}$ is in $\Pic(S)$ and \ref{prop:icmclifford:b} holds.
 
 \ref{prop:icmclifford:b}$\Rightarrow$\ref{prop:icmclifford:c}: This is obvious.
 
 \ref{prop:icmclifford:c}$\Rightarrow$\ref{prop:icmclifford:a}: Write $\ICM(R)=\bigsqcup_e G_e$, where $e$ runs over the set of idempotent elements of $\ICM(R)$, and $G_e$ denotes the group with unit $e$.
 Let $J$ be a fractional $R$-ideal representing $e$.
 Then there exists $x\in K^\times$ such that $xJ^2=J$. Put $S=xJ$. Then
 \[S^2=x^2J^2=x(xJ^2)=xJ=S.\]
 Note that $S$ is another representative of the class $e$ and by Lemma \ref{lemma:idempfracid} it is an over-order of $R$.
  Now let $T$ be any over-order of $R$.
  We want to show that $T^t$ is invertible in $T$.
  Say that the class representing $T^t$ lies in $G_e$ where $e=\isoclass{S}$.
  Then $T^t$ is invertible in $S$ and, since the multiplicator ring of $T^t$ is $T$, by Lemma \ref{lemma:invidealmultring}, we have that $S=T$.
\end{proof}

\begin{remark}
\label{rmk:computePic}
  If $K=K_1\times \ldots \times K_r$, with $K_i$ number fields, then $\cO_K=\prod_i \cO_{K_i}$ and
  \[\Pic(\cO_K) = \Pic(\cO_{K_1}) \times \ldots \times \Pic(\cO_{K_r}). \]
  There are well known algorithms to compute each $\Pic(\cO_{K_i})$, see \cite{psh08}.
  Note that if $S$ is an over-order of $R$, then the extension map $I\mapsto IS$ induces a surjective group homomorphism $\Pic(R)\twoheadrightarrow \Pic(S)$, see for example \cite[Corollary 2.1.11]{dadetz62}.
  In particular, if $S=\cO_K$ we have an exact sequence 
  \begin{equation}
  \label{eq:Picexactseq}
  0 \to R^\times \to \cO_K^\times\to \dfrac{\left( \cO_K / \frf \right)^\times}{\left( R / \frf \right)^\times} \to \Pic(R) \to \Pic(\cO_K)\to 0,
  \end{equation}
  where $\frf$ is the \emph{conductor} of $R$. 
  The exactness of \eqref{eq:Picexactseq} is classical for the case when $r=1$, that is, when $K$ is a number field. 
  A proof for the case $r>1$ can be found in \cite{JP16}.
  The results contained in \cite{klupau05} describe how to compute the middle term of \eqref{eq:Picexactseq} in the case $r=1$ and they can be extended word-by-word to the general case.
  Since $\cO_K^\times=\prod_i\cO^\times_{K_i}$ and there are well known algorithms to compute each $\cO^\times_{K_i}$, we deduce that we can effectively compute $\Pic(R)$ and $R^\times$.
\end{remark}


\section{Weak equivalence classes}
\label{sec:wkclasses}


The following result was proved in \cite{dadetz62} in the particular case of an integral domain.
We include a proof for completeness.
\begin{prop}
\label{prop:weakeq}
 Let $I$ and $J$ be two fractional $R$-ideals. The following are equivalent:
 \begin{enumerate}[(a)]
  \item \label{prop:weakeq:1}$I_\p$ and $J_\p$ are isomorphic for every prime $\p$ of $R$;
  \item \label{prop:weakeq:2} $1\in (I:J)(J:I)$;
  \item \label{prop:weakeq:3} $I$ and $J$ have the same multiplicator ring, say $S$, and there exists an invertible fractional $S$-ideal $L$ such that $I=LJ$.
 \end{enumerate}
\end{prop}
\begin{proof}
 \ref{prop:weakeq:1}$\Rightarrow$\ref{prop:weakeq:2}: Let $\p$ be a prime of $R$. By Corollary \ref{cor:isolocal} there exists a non-zero-divisor $x$ in the total quotient ring of $R_\p$ such that $I_\p=xJ_\p$, which in turn implies that
 \[(I_\p:J_\p)=(xJ_\p:J_\p)=x(J_\p:J_\p)\quad\text{and}\quad (J_\p:I_\p)=(J_\p:xJ_\p)=\frac 1x(J_\p:J_\p). \]
 Therefore
 \[ ((I:J)(J:I))_\p=(I_\p:J_\p)(J_\p:I_\p)=x(J_\p:J_\p)\frac 1x(J_\p:J_\p)=(J_\p:J_\p), \]
 which clearly contains $1$.
 Hence, the natural inclusion $(J:I)(I:J)\subseteq (J:J)$ is locally surjective at $\p$. Since the choice of $\p$ was arbitrary we conclude that $(J:I)(I:J) = (J:J)$ and in particular that $1\in (J:I)(I:J)$. 
 
 \ref{prop:weakeq:2}$\Rightarrow$\ref{prop:weakeq:3}:  
 By definition of quotient ideal we have that $(I:J)(J:I)\subseteq (I:I)$ and that $(I:J)(J:I)\subseteq (J:J)$.
 Since $(I:J)(J:I)$ has a structure of both $(I:I)$ and $(J:J)$-module and contains $1$, it follows that
 \[ (I:I)=(I:J)(J:I)=(J:J), \]
 that is, $I$ and $J$ have the same multiplicator ring and $(I:J)$ and $(J:I)$ are inverse to each other.
 The following inclusions
 \[ I = I(I:I)=I(I:J)(J:I)\subseteq J(I:J)\subseteq I\]
 are therefore equalities and in particular $I=LJ$ for $L=(I:J)$.
 
 \ref{prop:weakeq:3}$\Rightarrow$\ref{prop:weakeq:1}: Let $L'$ be any invertible ideal in $R$ such that $L'S=L$.
 Note that such an $L'$ exists since the extension map $\Pic(R)\to \Pic(S)$ is surjective, as we explain in Remark~\ref{rmk:computePic}.
 The localization $L'_\p$ at any prime $\p$ of $R$ is principal by Lemma \ref{lemma:invertilbeprincipal}, say $L'_\p=xR_\p$. Then $I_\p=xJ_\p$ and hence $I_\p \simeq J_\p$.
\end{proof}

\begin{df}
   If two fractional $R$-ideals $I$ and $J$ satisfy the equivalent conditions of Proposition \ref{prop:weakeq} we say that they are \emph{weakly equivalent}.
   Denote by $\Wk(R)$ the set of weak equivalence classes and by $\wkclass{I}$ the weak equivalence class of a fractional $R$-ideal $I$.
   Given any over-order $S$ of $R$ let $\bWk{S}(R)
   $ be the subset of $\Wk(R)$ consisting of the weak equivalence classes $\wkclass{I}$ such that $(I:I)=S$.
\end{df}
Note that $\Wk(R)$ inherits the structure of a commutative monoid from $\cI(R)$.
Consider the partition
\[\Wk(R) = \bigsqcup \bWk{S}(R),\]
where the disjoint union is taken over all the over-orders $S$ of $R$.
By Proposition \ref{prop:weakeq}.\ref{prop:weakeq:2} an ideal is invertible if and only if it is weakly equivalent to its multiplicator ring and hence we have that $\bWk{S}(R)=\set{\wkclass{S}}$ if and only if $S$ is Gorenstein.

\begin{remark}
   Let $p$ be a rational prime number and put $R_{(p)}=R\otimes_\Z \Z_{(p)}$. Similarly, for fractional $R$-ideals $I$ and $J$, put $I_{(p)}:=I\otimes_\Z \Z_{(p)}$ and $J_{(p)}:=J\otimes_\Z \Z_{(p)}$. The ideals $I$ and $J$ are said to belong to the same \emph{genus} if and only if $I_{(p)}$ and $J_{(p)}$ are isomorphic as $R_{(p)}$-modules for every rational prime $p$.
   Note that $R_{(p)}$ is a semi-local ring and hence by Lemma \ref{lemma:invertilbeprincipal} fractional ideals are invertible if and only if they are principal and generated by a non-zero-divisor. An easy modification of the proof of Proposition \ref{prop:weakeq} shows that $I$ and $J$ are weakly equivalent if and only if they belong to the same genus. This equivalence was already noticed in \cite[Section 5]{LevyWiegand85}.
   The notion of genus is classical and it has been widely studied in the literature, see for example \cite[Section 7]{Reiner03} and \cite[Section 6]{Reiner70}.
   For example, it is known that $I$ and $J$ are in the same genus if and only if they are isomorphic after tensoring with the $p$-adic completion $\Z_p$, which in turn holds if and only if the quotients $I/p^k I$ and $J/p^k J$ are isomorphic for an integer $k$, that only depends on $R_{(p)}$.
   We prefer to work with the notion of weak equivalence introduced above, since part \ref{prop:weakeq:2} of Proposition \ref{prop:weakeq} implies that checking whether two ideals are weakly equivalent can be performed in polynomial time.
\end{remark}

\begin{cor}
 Let $I$ and $J$ be fractional ideals. Then
 \[ \isoclass{I}=\isoclass{J} \Longleftrightarrow \isoclass{I^t}=\isoclass{J^t} \]
 and
 \[ \wkclass{I}=\wkclass{J} \Longleftrightarrow \wkclass{I^t}=\wkclass{J^t}. \]
\end{cor}
\begin{proof}
 Note that $I=xJ$ if only if $I^t=(1/x)J^t$, which gives the first equivalence.
 The second equivalence follows from the equality $(I:J)(J:I)=(J^t:I^t)(I^t:J^t)$ and part \ref{prop:weakeq:2} of Proposition \ref{prop:weakeq}.
\end{proof}

Note that two invertible fractional $R$-ideals $I$ and $J$ are isomorphic if and only if $IJ^{-1}$ is a principal fractional $R$-ideal.
Since being weakly equivalent is a necessary condition for being isomorphic, we can also reduce the isomorphism problem between non-invertible ideals to a principal ideal problem.

\begin{cor}
\label{cor:isom_test}
 Let $I$ and $J$ be two weakly equivalent fractional $R$-ideals, and let $S$ be their multiplicator ring. Then
 \[I = (I:J) J,\]
 and $(I:J)$ is an invertible fractional $S$-ideal.
 In particular, $I\simeq J$ if and only if $(I:J)$ is a principal fractional $S$-ideal.
\end{cor}
\begin{proof}
 Let $S$ be the multiplicator ring of $I$ and $J$.
 If $I$ and $J$ are weakly equivalent, we show in the proof of Proposition \ref{prop:weakeq} that $(I:J)$ is an invertible fractional $S$-ideal and $I=(I:J)J$. 
 In particular, if $(I:J)$ is a principal fractional $S$-ideal, then $I$ and $J$ are isomorphic.
 
 Conversely, if $I=xJ$ for some $x\in K^\times$ then
 \[ (I:J)=(xJ:J)=x(J:J)=xS, \]
 which concludes the proof.
\end{proof}

Finally, knowing the weak equivalence classes allows us to reconstruct the isomorphism classes.
Let $S$ be an over-order of $R$ and define
\[\bICM{S}(R)=\set{\isoclass{I}\in \ICM(R) \text{ s.t. } (I:I)=S},\]
so that we get
\[\ICM(R)=\bigsqcup \bICM{S}(R),\]
where the disjoint union is taken over all the over-orders $S$ of $R$.

\begin{thm}
\label{thm:icmcomputation}
 Let $R$ be an order in $K$. For every over-order $S$ of $R$, the action of $\Pic(S)$ on $\bICM{S}(R)$ induced by ideal multiplication is free and
 \[\bWk{S}(R)=\faktor{\bICM{S}(R)}{\Pic(S)}.\]
 More concretely, if
 \[\bWk{S}(R)=\set{\wkclass{I_1},\ldots,\wkclass{I_r}} \qquad \text{ and } \qquad \Pic(S)=\set{\isoclass{J_1},\ldots,\isoclass{J_s}},\]
 with the $I_i$'s pairwise not weakly equivalent and the $J_j$'s pairwise not isomorphic then
 \[\bICM{S}(R)=\set{\isoclass{I_iJ_j}:1\leq i \leq r, 1\leq j \leq s}\]
 and the fractional ideals $I_iJ_j$ are pairwise not isomorphic.
\end{thm}
\begin{proof}
 Let $I$ be a fractional $R$-ideal with multiplicator ring $S$. 
 Then $\wkclass{I}=\wkclass{I_i}$ for some $i$, that is, there exists an invertible fractional $S$-ideal $J$ such that $I=I_iJ$.
 Let $j$ be the index such that $\isoclass J = \isoclass{J_j}$. 
 It follows that $\isoclass{I}=\isoclass{I_i J_j}$. 
 It remains to prove that if $\isoclass{I_iJ_j}=\isoclass{I_kJ_h}$, that is, $I_iJ_j=xI_kJ_h$ for some $x\in K^\times$, then $i=k$ and $j=h$.
 Multiplying by $(S:J_j)$ on both sides, we get by Proposition \ref{prop:weakeq}.\ref{prop:weakeq:3} that $I_i$ is weakly equivalent to $I_k$, that is, $i=k$.
 To conclude, it is enough to prove that if $I=IJ$ with $I$ and $J$ both having multiplicator ring $S$ and $J$ invertible, then $J=S$.
 We will prove that this is true locally at every prime of $S$.
 Since $J$ is invertible, we have by Lemma \ref{lemma:invertilbeprincipal} that $J_\p=yS_\p$ for some non-zero-divisor $y$.
 Hence it follows that
 \[I_\p=yS_\p I_\p=yI_\p\]
 which implies that both $y$ and $1/y$ are in $(I_\p:I_\p)=S_\p$. Therefore we again have $J_\p=S_\p$.
\end{proof}

\begin{remark}
 Fixing the multiplicator ring is a key point in using the previous proposition.
 Let $R=\Z[\alpha]$, where $\alpha$ is a root of $f(x)=x^2-8x-8$. Note that $\cO_K=\Z[\alpha/2]$ and $[\cO_K:R]=2$.
 Consider the invertible $R$-ideal $\p=(3,\alpha-1)$ and the conductor $\frf=(2,\alpha)$ of $R$.
 It is easy to verify that $\Pic(\cO_K)$ is trivial, while $\Pic(R)\simeq \Z/2\Z$ with generator the ideal class of $\p$.
 It follows that the product $\frf \p $ is isomorphic to $\frf$ and, in particular, that the action of $\Pic(R)$ on the whole $\ICM(R)$ is not free.
\end{remark}

Using Theorem \ref{thm:icmcomputation} we can compute the ideal class monoid of an order $R$ if we know all its over-orders, their Picard groups and the weak equivalence class monoid.
For the first issue, by Lemma \ref{lemma:idempfracid}, it is enough to look at the idempotent $R$-modules of the finite quotient $\cO_K/R$.
In the end of Section \ref{sec:idealclasses} we discussed how to compute the Picard group of a possibly non-maximal order.
Finally, in the next section we will describe how to compute $\Wk(R)$.
See Section \ref{sec:algorithms} for the corresponding algorithms.


\section{Computing the weak equivalence class monoid}
\label{sec:computingwkclasses}


The following results are inspired by \cite{dadetz62} where the authors produce similar results in the particular case of an integral domain.
Let $R$ be an order in $K$.
Recall that we can partition $\Wk(R)$ as the disjoint union of $\bWk{S}(R)$ where $S$ runs through the set of over-orders of $R$.
We will now describe a method to compute $\bWk{S}(R)$.
Observe that when $S$ is not Gorenstein there are always at least two distinct classes in $\bWk{S}(R)$, namely $\wkclass{S}$ and $\wkclass{S^t}$.
\begin{prop}
\label{prop:wkbar}
 Let $T$ be any over-order of $S$ such that $S^tT$ is invertible as a fractional $T$-ideal.
 Let $\frf$ be an ideal contained in $S$ such that $T \subseteq (\frf:\frf)$.
 Then every class in $\bWk{S}(R)$ has a representative $I$ satisfying $\frf \subseteq I \subseteq T$.
\end{prop}
\begin{proof}
 Let $I'$ be any fractional ideal with $(I':I')=S$.
 By Lemma \ref{tracedualmultiplicatorring} we have that $I'(I')^tT=S^tT$ and hence it follows that $I'T$ is an invertible fractional $T$-ideal.
 Let $J$ be a representative of the pre-image under the surjective map $\Pic(S)\to\Pic(T)$ of the class of $(T:I'T)$ and put $I=I'J$.
 Note that $\wkclass{I'}=\wkclass{I}$ in $\bWk{S}(R)$ and that $IT=T$,
 which implies that $I\subseteq T$.
 
 On the other hand, as $\frf T= \frf$ we get that
 \[\frf I = \frf T I = \frf T = \frf, \]
 and, since $\frf \subseteq (I:I)$, we obtain that $\frf=\frf I \subseteq I$, and we can conclude that $\frf \subseteq I \subseteq T$.
\end{proof}
The previous proposition tells us that in order to compute the representatives of  $\bWk{S}(R)$ we can look at the sub-$S$-modules of the finite quotient $T/\frf$.
One possible choice is to take $T=\cO_K$ and $\frf=(S:\cO_K)$, but to gain in efficiency we want to keep the quotient as small as possible.
The natural choice is to take as $T$ the smallest over-order of $S$ with $S^tT$ invertible (as a fractional $T$-ideal) and as $\frf$, the colon ideal $(S:T)$, which is the biggest fractional $T$-ideal in $S$.

\begin{remark}
    Given orders $S\subseteq T$, let $\frf=(S:T)$.
    If $S$ is Gorenstein then by Lemma \ref{lemma:quotidealprod} and Proposition \ref{prop:eqGorenstein} it follows that 
    \[ (\frf:\frf)= (S:T(S:T))=(S:(S:T))=T. \]
    If $S$ is not Gorenstein then the multiplicator ring of $\frf$ might still be equal to $T$. This for example must be the case when $T=\cO_K$. The multiplicator ring of $\frf$ can also be strictly bigger than $T$, as Example \ref{ex:multringconductor} shows. 
    If we assume that $S^tT$ is invertible in $T$, as required in Proposition \ref{prop:wkbar}, then $(\frf:\frf)=T$, because $\frf=(S^tT)^t$ and a fractional ideal has the same multiplicator ring as its trace dual.
\end{remark}

\begin{example}
\label{ex:multringconductor}
    Let $f=x^6 - 4x^5 + 11x^4 - 24x^3 + 55x^2 - 100x + 125$.
    Observe that $f=f_1f_2f_3$ where $f_1=x^2 - 4x + 5$, $f_2=x^2 - 2x + 5$ and $f_3=x^2 + 2x + 5$.
    Put $K=\Q[x]/(f)$ and $K_i=\Q[x]/(f_i)$ for $i=1,2,3$.
    We identify $K=K_1\times K_2\times K_3$ and, for $i=1,2,3$, we denote by $e_i$ and $\alpha_i$ the elements $1\bmod f_i$ and $x\bmod f_i$, respectively.
    Consider the orders
    \[ S=e_1\Z\oplus \alpha_1\Z \oplus 2e_2\Z  \oplus \left(\frac 12 e_2+\frac 12 \alpha_2\right)\Z \oplus (e_2+e_3)\Z \oplus (e_2+\alpha_3)\Z \]
    and
    \[ T=e_1\Z\oplus \alpha_1\Z \oplus e_2 \Z\oplus \left(\frac 12 e_2+\frac 12 \alpha_2\right)\Z \oplus e_3\Z\oplus \alpha_3\Z.\]
    Then $S\subseteq T$ with index $2$ and the multiplicator ring of $\frf=(S:T)$ is the maximal order 
    \[ \cO_K=e_1\Z\oplus \alpha_1\Z \oplus e_2 \Z\oplus \left(\frac 12 e_2+\frac 12 \alpha_2\right)\Z \oplus e_3\Z\oplus \left(\frac 12 e_3+\frac 12 \alpha_3\right)\Z \]
    and it is easy to check that $[\cO_K:T]=2$.
\end{example}

\begin{remark}
    Let $I'$ be a fractional $R$-ideal. As in the proof of Proposition \ref{prop:wkbar}, let $J$ be a representative of the pre-image under the extension map $\Pic(R)\twoheadrightarrow \Pic(\cO_K)$ of $\isoclass{ (\cO_K:I'\cO_K) }$ and put $I=I'J$. Then $\wkclass{I}=\wkclass{I'}$ and $I\cO_K=\cO_K$ which implies
    \[ \frf\subseteq I \subseteq \cO_K, \]
    where $\frf$ is the conductor of $R$, that is, $\frf=(R:\cO_K)$.
    So if the quotient $\cO_K/\frf$ is not too big we can look directly at its sub-$R$-modules in order to get all representatives of the classes of $\Wk(R)$.
    One can also obtain all the over-orders of $R$ by computing the multiplicator rings of the representatives of $\Wk(R)$.
\end{remark}

Let $T$ be an over-order of $S$ such that $S^tT$ is an invertible fractional $T$-ideal.
Choose primes $\p_1,\ldots,\p_r$ of $S$ and positive integers $e_1,\ldots, e_r$ such that $T \subseteq (\frf:\frf)$, where
\[\frf=\p_1^{e_1}\cdots \p_r^{e_r}.\]
Note that such $\frf$ satisfies the hypothesis of Proposition \ref{prop:wkbar} and, moreover, $\frf\subset (S:T)$, since $(S:T)$ contains all fractional $T$-ideals contained in $S$.
It follows that the primes $\p_i$ must be non-invertible.

By the Chinese Remainder Theorem there is a ring isomorphism
\[ \dfrac{S}{\frf} \simeq \dfrac{S}{\p_1^{e_1}}\times \ldots \times \dfrac{S}{\p_r^{e_r}}, \]
which, after taking the tensor product with $T$, becomes
\begin{equation}
  \dfrac{T}{\frf} \simeq \dfrac{T}{\p_1^{e_1}T}\times \ldots \times \dfrac{T}{\p_r^{e_r}T}. \label{eq:localwkclasses}
\end{equation}
Observe that the isomorphism \eqref{eq:localwkclasses} is compatible with ideal multiplication and hence it respects weak equivalences.
In particular, we can compute $\bWk{S_{\p_i}}(S_{\p_i})$ by looking at the sub-$S$-modules of the ``local'' quotient $T/\p_i^{e_i}T$ up to weak equivalence.
Then we can ``patch'' them together via the isomorphism \eqref{eq:localwkclasses} and hence reconstruct all the representatives of $\bWk{S}(R)$.
If $r>1$ this tells us that we can split the computation of $\bWk{S}(R)$ and hence potentially obtain a more efficient algorithm.
The next two remarks will tell us that we can further improve the algorithm by ignoring or reducing some factors in \eqref{eq:localwkclasses} if the corresponding primes $\p_i$ satisfy certain conditions.

\begin{remark}
      Let $\p_i$ be one the primes appearing in \eqref{eq:localwkclasses}.
      If the $S/\p_i$-vector space $S^t/\p_iS^t$ is one-dimensional, then by Nakayama's lemma we have that $S^t$ is locally principal at $\p_i$.
      It follows that each fractional ideal $I$ with multiplicator ring $S$, that is, $II^t=S^t$, will be locally invertible at $\p_i$, or, in other words, $\bWk{S_{\p_i}}(S_{\p_i})$ is trivial.
\end{remark}

\begin{remark}
      Let $\p$ be one of the primes appearing in the decomposition \eqref{eq:localwkclasses}.
      Observe that $T_{\p}$ has only finitely many primes $\frP_1,\ldots,\frP_m$, which are exactly the ones lying above $\p$.
      Assume that $m<q$, where $q=\#(S/\p)$. Then by \cite[Lemma 4]{delcdvor00} for each ideal $I$ of $S_\p$ such that $IT_\p$ is invertible there exists $x\in I$ such that $IT_\p=xT_\p$.
      This implies that 
      \[ S_\p \subseteq \frac{1}{x}I\subseteq \frac{1}{x}IT_\p=T_\p. \]
      This means that, if we also assume that $S^tT$ is invertible in $T$, we can find all the classes of $\bWk{S_\p}(S_\p)$ in the quotient $T_\p/S_\p$ and this quotient might be smaller than $T/\p_i^{e_i}T$.
\end{remark}


\section{ Algorithms }
\label{sec:algorithms}

In this section we present the pseudo-code for the algorithms described in the previous sections.
The implementation in Magma \cite{Magma} is available at \url{https://github.com/stmar89/AbVarFq}.
We will use without mentioning well known algorithms for abelian groups, which can all be found in \cite[Section 2.4]{cohen93}.\\

\begin{algorithm}[H]

 \KwIn{An order $R$ in a $\Q$-\'etale algebra $K$;}
 \KwOut{A list $\cL^{o}$ containing the over-orders of $R$;}
 Compute the maximal order $\cO_K$ of $K$\;
 Compute the quotient as abelian groups $q:\cO_K\twoheadrightarrow Q:=\cO_K/R$\;
 Initialize an empty list $\cL^{o}$\;
 
 \For{$\text{each }H' \unlhd Q $}{
    Put $S:=\Span{q^{-1}(H')}_R$\;
    \If{ $SS=S$ and $S \not \in \cL^{o}$ }{Append $S$ to $\cL^{o}$\;}
 } 
 \KwRet{$\cL^{o}$}\;
 \caption{\label{alg:over-orders} Computing over-orders of a given order}
\end{algorithm}
\begin{thm}
 Algorithm \ref{alg:over-orders} is correct.
\end{thm}
\begin{proof}
 This follows from the fact that the over-orders of $R$ are precisely the idempotent fractional $R$-ideals contained in $\cO_K$ and containing $R$, as shown in Lemma \ref{lemma:idempfracid}.
\end{proof}

A new and much more efficient version of Algorithm \ref{alg:over-orders} is described in \cite{HofmannSircana19}.
\\

\begin{algorithm}[H]
 \KwIn{Two fractional $R$-ideals $I$ and $J$;}
 \KwOut{Whether $I$ and $J$ are weakly equivalent;}
 \eIf{$1\in (I:J)(J:I)$}
     {\KwRet{true}\;}
     {\KwRet{false}\;}
 \caption{\label{alg:wkeq_test} Returns whether two fractional $R$-ideals $I$ and $J$ are weakly equivalent}
\end{algorithm}

\begin{thm}
 Algorithm \ref{alg:wkeq_test} is correct.
\end{thm}
\begin{proof}
    It follows by Proposition \ref{prop:weakeq}.
\end{proof}

\begin{algorithm}[H]
 \KwIn{An order $S$ in a $\Q$-\'etale algebra $K$;}
 \KwOut{A list $\cL^{w}$ of the representatives of the weak equivalence classes of ideals with endomorphism ring $S$, that is $\bWk{S}(R)$;}
 Compute the trace dual ideal $S^t$\;
 Initialize an empty list $\cL^{w}$\;
 \eIf{ $1 \in S^t(S:S^t)$ }
 {Append $S$ to $\cL^{w}$\;}
 {
      Find an over-order $T$ of $S$ such that $1 \in S^tT(T:S^tT)$  \tcp*[l]{use Algorithm \ref{alg:over-orders}}
      Put $\frf:=(S:T)$\;
      Consider the quotient $q:=T \twoheadrightarrow Q:=T/\frf$\;
      \For{$\text{each }H' \unlhd Q $}
      {
	  Put $I:=\Span{q^{-1}(H')}_S$\;
	  \If(\tcp*[h]{use Algorithm \ref{alg:wkeq_test}}){there is no $J \in \cL^{w}$ weakly eq.~to $I$}
	  {Append $I$ to $\cL^{w}$\;}
      } 
 } 
 \KwRet{$\cL^{w}$}\;
 \caption{\label{alg:wkbar} Computing representatives of the classes in $\bWk{S}(R)$ for an order $S$}
\end{algorithm}
\begin{thm}
 Algorithm \ref{alg:wkbar} is correct.
\end{thm}
\begin{proof}
 The correctness of the algorithm follows from Propositions \ref{prop:weakeq} and \ref{prop:wkbar}.
\end{proof}

\begin{algorithm}[H]
 \KwIn{An order $R$ in a $\Q$-\'etale algebra $K$;}
 \KwOut{A list $\cL^{iso}$ of the representatives of the isomorphism classes of ideals, that is $\ICM(R)$;}
 Compute the over-orders $\cL^{o}$ of $R$  \tcp*[l]{use Algorithm \ref{alg:over-orders}}
 Initialize the empty list $\cL^{iso}$\;
 \For{each $S$ in $\cL^{o}$}
 {
 Compute a list $\cL^{w}_S$ of representatives of $\bWk{S}(R)$ \tcp*[l]{use Algorithm \ref{alg:wkbar}}
 Compute a list $\cL^{i}_S$ of representatives of $\Pic(S)$\;
 \For{each $I$ in $\cL^{w}_S$ and each $J$ in $\cL^{i}_S$}
 {Append $IJ$ to $\cL^{iso}$\;}
 }
 \KwRet{$\cL^{iso}$}\;
 \caption{\label{alg:icm} Computing representatives of the classes in $\ICM(R)$ for an order $R$}
\end{algorithm}
\begin{thm}
 Algorithm \ref{alg:icm} is correct.
\end{thm}
\begin{proof}
 This follows from Theorem \ref{thm:icmcomputation}.
\end{proof}

\begin{algorithm}[H]
 \KwIn{Two fractional $R$-ideals $I$ and $J$;}
 \KwOut{Whether $I$ and $J$ are isomorphic and if so it returns an element $\alpha\in K$ such that $\alpha I=J$; }
 \If(\tcp*[h]{use Algorithm \ref{alg:wkeq_test}}){$I$ and $J$ are weakly equivalent}
     {Compute $C=(J:I)$\;
      Compute $S=(I:I)$\;
      \If{ $C$ is a principal $S$-ideal}
      	  {Compute $\alpha$ such that $\alpha S=C$\;
       	  \KwRet{(true, $\alpha$)}\;
      	  }
     }
 \KwRet{false}\;
 \caption{\label{alg:isom_test} Returns whether two fractional $R$-ideals $I$ and $J$ are isomorphic and if so it returns an element $\alpha\in K$ such that $\alpha I=J$}
\end{algorithm}

\begin{thm}
 Algorithm \ref{alg:isom_test} is correct.
\end{thm}
\begin{proof}
 This follows from Corollary \ref{cor:isom_test}.
\end{proof}

\begin{algorithm}[H]
 \KwIn{a list $\cL^{iso}$ of representatives of $\ICM(R)$ and a fractional $R$-ideal $I$;}
 \KwOut{a pair $(I_0,\alpha)$, consisting of a fractional $R$-ideal $I_0$ from the list $\cL^{iso}$ and an element $\alpha\in K^\times$ such that $\alpha I=I_0$;}
 Put $\cL_{I}=\set{ J \in \cL^{iso} : J \text{ is weakly eq.~to } I }$ \tcp*[l]{use Algorithm \ref{alg:wkeq_test}}
 Scan $\cL_{I}$ for $I_0$ such that $\alpha I=I_0$, for some $\alpha \in K^\times$ \tcp*[l]{use Algorithm \ref{alg:isom_test}}
 \KwRet{$(I_0,\alpha)$}\;
 \caption{\label{alg:icm_dlp} Returns the representative of the isomorphism class of a given fractional $R$-ideal }
\end{algorithm}

\begin{thm}
 Algorithm \ref{alg:icm_dlp} is correct.
\end{thm}
\begin{proof}
 The algorithm terminates because $\cL_{I}$ contains by construction a fractional $R$-ideal isomorphic to $I$.
\end{proof}

\begin{remark}
\label{rmk:icm_dlp_var}
If the list $\cL^{iso}$ is computed using Algorithm \ref{alg:icm} and in the process one stores the lists $\cL_S^w$ and $\cL_S^i$ then Algorithm \ref{alg:icm_dlp} could be modified as follows.
Firstly, one computes the multiplicator ring $S$ of $I$, then one finds a representative $I_w$ of the weak equivalence of $I$ in $\cL_S^w$ using Algorithm \ref{alg:wkeq_test} and then one finds a representative $I_i$ of the class of $(I:I_w)$ in $\cL_S^i$ either by solving the corresponding DLP in $\Pic(S)$ or by using Algorithm \ref{alg:isom_test}.
Finally the representative $I_0$ of the class of $I$ in $\cL^{iso}$ is given by $I_wI_i$.
\end{remark}

\begin{remark}
\label{rmk:mult_table}
Given $I$ and $J$ in $\cL^{iso}$ one can use Algorithm \ref{alg:icm_dlp} (or the variation in Remark \ref{rmk:icm_dlp_var}) to compute the representative in $\cL^{iso}$ of the product $IJ$.
 This means that we can build a multiplication table of the commutative monoid $\ICM(R)$.
\end{remark}

\begin{remark}
Algorithm \ref{alg:over-orders} requires the enumeration of the subgroups of the finite quotient $\cO_K/R$ which is a very costly computation.
Since the number of subgroups that give rise to over-orders of $R$ is much smaller than the total number, there is a big room for improvement.
A new and efficient solution is provided in \cite{HofmannSircana19}.
Similar considerations are valid also for the computations of the weak equivalence classes in Algorithm \ref{alg:wkbar}.
In particular these two algorithms constitute the major bottlenecks of the computation of $\ICM(R)$ in Algorithm \ref{alg:icm}.
Algorithm \ref{alg:wkeq_test} is efficient because it involves only the computations of ideal quotients, product of fractional ideals and testing inclusion.
The most expensive operation in Algorithms \ref{alg:isom_test} and \ref{alg:icm_dlp} is testing whether a fractional ideal is principal and producing a generator.
\end{remark}

\section{Examples}
\label{sec:examples}

The examples contained in this section were computed with Magma \cite{Magma}.
The implementation of the algorithms from Section \ref{sec:algorithms} can be found at \url{https://github.com/stmar89/AbVarFq}.

\begin{example}
\label{ex1}
 Let $f=x^3+31x^2+43x+77$ and let $\alpha$ be a root of $f$. Consider the monogenic order defined by $f$, say $E=S_1=\Z[\alpha]$. There are $15$ over-orders of $E$. The maximal order is $\cO=S_{15}=\Z\oplus\frac{\alpha+5}{8}\Z\oplus \frac{\alpha^2 + 2\alpha + 49}{64}\Z$. Observe that $[\cO:E]=512$, so the only singular prime is $2$.
 In Figure \ref{ex1:fig} and Table \ref{ex1:table} we describe the over-orders with the weak equivalence classes and Picard groups.

  \begin{figure}[h]
  	\scalebox{1}[1]{
  	\xymatrix@C=10pt{
						&						&											&											& (S_9,2) \ar[rd]^{2}	&												& (S_{13},1) \ar[rd]^{2}	&& \\
						&						&											& (S_5,2) \ar[r]^{2}\ar[ru]^{2}\ar[rd]^{2}	& (S_8,2) \ar[r]^{2}		& (S_{10},2) \ar[r]^{2}\ar[ru]^{2}\ar[rd]_{2} 	& (S_{12},1) \ar[r]^{2}	& (S_{14},2) \ar[r]^{2^2}	& (S_{15},1)\\ 
	(S_1,1) \ar[r]^{2} 	& (S_2,2) \ar[r]^{2} 	& (S_3,2) \ar[r]^{2}\ar[ru]^{2}\ar[rd]_{2} 	& (S_6,1) \ar[r]^{2}							& (S_7,2) \ar[ru]_{2}	&												& (S_{11},1) \ar[ru]_{2}	&& \\
						&						&											& (S_4,1) \ar[ru]_{2}						&&&&& \\ 	
  	}
  	}
    \caption{This is the lattice of inclusions of the over-orders of $S_1$  from Example \ref{ex1}.
    Each vertex is labeled as $(S_i,\#\bWk{S_i}(S_1))$. The edges are marked by the index of the corresponding inclusion.
    \label{ex1:fig}}
  \end{figure}

  \begin{table}[h]
    	\scalebox{.9}[.9]{
    \begin{tabular}{cccc}
	\hline
	i & $\Z$-basis of $S_i$ 						& $[\cO:S_i]$ & $\Pic(S_i)$ \\ \hline
	1 & $1,\alpha,\alpha^2$ 						& $512$ & $\nicefrac{\Z}{4\Z}$  \\ \hline
	2 & $1 ,\alpha,\frac{\alpha^2 + 1}{2} $ 				& $256$ & $\nicefrac{\Z}{4\Z}$\\ \hline
	3 & $1 ,\alpha,\frac{\alpha^2 + 2\alpha + 1}{4} $ 			& $128$ & $\nicefrac{\Z}{4\Z}$ \\ \hline
	4 & $1 ,\alpha,\frac{\alpha^2 + 6\alpha + 5}{8} $ 			& $64$ 	& $\nicefrac{\Z}{2\Z}$ \\ \hline
	5 & $1 ,\alpha,\frac{\alpha^2 + 2\alpha + 1}{8} $			& $64$  & $\nicefrac{\Z}{4\Z}$ \\ \hline
	6 & $1 ,\frac{\alpha + 1}{2},\frac{\alpha^2 + 3}{4}$  		& $64$  & $\nicefrac{\Z}{2\Z}$ \\ \hline
	7 & $1 ,\frac{\alpha + 1}{2},\frac{\alpha^2 + 2\alpha + 1}{8}$ 	& $32$  & $\nicefrac{\Z}{2\Z}$\\ \hline
	8 & $1 ,\alpha,\frac{\alpha^2 + 10\alpha + 9}{16}$ 			& $32$ 	& $\nicefrac{\Z}{2\Z}$\\ \hline
	9 & $1 ,\alpha,\frac{\alpha^2 + 2\alpha + 1}{16}$ & $32$ 	& $\nicefrac{\Z}{4\Z}$\\ \hline
	10 & $1 ,\frac{\alpha + 1}{2},\frac{\alpha^2 + 2\alpha + 1}{16}$ 	& $16$ 	& $\nicefrac{\Z}{2\Z}$ \\ \hline
	11 & $1 ,\frac{\alpha + 1}{2},\frac{\alpha^2 + 2\alpha + 17}{32}$	& $8$ 	& $1$ \\ \hline    
	12 & $1 ,\frac{\alpha + 1}{2},\frac{\alpha^2 + 10\alpha + 25}{32}$	& $8$ 	& $1$ \\ \hline    
	13 & $1 ,\frac{\alpha + 1}{4},\frac{\alpha^2 + 2\alpha + 1}{16}$	& $8$ 	& $\nicefrac{\Z}{2\Z}$ \\ \hline    
	14 & $1 ,\frac{\alpha + 1}{4},\frac{\alpha^2 + 2\alpha + 17}{32}$	& $4$ 	& $1$ \\ \hline    
	15 & $1 ,\frac{\alpha + 5}{8},\frac{\alpha^2 + 2\alpha + 49}{64}$	& $1$ 	& $1$ \\ \hline    
    \end{tabular}
    }     
    \caption{The over-orders of $S_1$ from Example \ref{ex1}}
    \label{ex1:table}
  \end{table}
 Among the over-orders of $E$, the orders $S_2,S_3,S_5,S_7,S_8,S_9,S_{10},S_{14}$ are non-Gorenstein and there are no other non-invertible weak equivalence classes apart from $S_{i}^t$.
 This implies that $\# \Wk(E)=23$ and, using the information about the Picard groups, we can deduce that $\#\ICM(E)=59$.
\end{example}

\begin{example}
\label{ex2}
 Let $f=x^3-1000x^2-1000x-1000$ and let $\alpha$ be a root of $f$. Consider the monogenic order defined by $f$, say $E=S_1=\Z[\alpha]$. There are $16$ over-orders of $E$. The maximal order is $\cO=S_{16}=\Z\oplus\frac{\alpha}{10}\Z\oplus \frac{\alpha^2}{100}\Z$. Observe that $[\cO:E]=1000$, so the singular primes are $2$ and $5$.
 In Figure \ref{ex2:fig} and Table \ref{ex2:table} we describe the over-orders with the weak equivalence classes and Picard groups. 
  \begin{figure}[h]
  	\scalebox{1}[1]{
  	\xymatrix{
  	(S_1,1) \ar[d]_{5}\ar[r]^{2} & (S_2,2) \ar[d]_{5}\ar[r]^{2} 		& (S_3,1) \ar[d]_{5}\ar[r]^{2} 		& (S_{5},1) \ar[d]_{5} \\
  	(S_4,2) \ar[d]_{5}\ar[r]^{2} & (S_6,4) \ar[d]_{5}\ar[r]^{2} 		& (S_{7},2) \ar[d]_{5}\ar[r]^{2}		& (S_{9},2) \ar[d]_{5} \\
  	(S_8,1) \ar[d]_{5}\ar[r]^{2} & (S_{10},2) \ar[d]_{5}\ar[r]^{2} 	& (S_{11},1) \ar[d]_{5}\ar[r]^{2}	& (S_{13},1) \ar[d]_{5} \\
  	(S_{12},1) \ar[r]^{2}		& (S_{14},2)  \ar[r]^{2} 			& (S_{15},1)  \ar[r]^{2}				& (S_{16},1)\\
  	}
  	}
    \caption{This is the lattice of inclusions of the over-orders of $S_1$  from Example \ref{ex2}.
    Each vertex is labeled as $(S_i,\#\bWk{S_i}(S_1))$. The edges are marked by the index of the corresponding inclusion.}
    \label{ex2:fig}
  \end{figure}
  \begin{table}[h!]
  	\scalebox{.9}[.9]{
    \begin{tabular}{cccc}
	\hline
	i & $\Z$-basis of $S_i$ 						& $[\cO:S_i]$ & $\Pic(S_i)$ \\ \hline		
	1 & $1,\alpha,\alpha^2$ 						& $1000$ & $\nicefrac{\Z}{2\Z}\times \nicefrac{\Z}{8880\Z}$ 	\\ \hline 
	2 & $1 ,\alpha,\frac{\alpha^2}{2} $ 				& $500$  & $\nicefrac{\Z}{8880\Z}$ 			\\ \hline
	3 & $1 ,\alpha,\frac{\alpha^2 + 2\alpha}{4}$ 			& $250$  & $\nicefrac{\Z}{8880\Z}$ 			\\ \hline
	4 & $1 ,\alpha,\frac{\alpha^2}{5}$ 					& $200$	 & $\nicefrac{\Z}{2\Z}\times\nicefrac{\Z}{1776\Z}$ 	\\ \hline
	5 & $1 ,\frac{\alpha}{2},\frac{\alpha^2}{4}$			& $125$  & $\nicefrac{\Z}{2960\Z}$ 			\\ \hline 
	6 & $1 ,\alpha,\frac{\alpha^2}{10}$  				& $100$  & $\nicefrac{\Z}{1776\Z}$ 			\\ \hline 
	7 & $1 ,\alpha,\frac{\alpha^2 + 10\alpha}{20}$			& $50$   & $\nicefrac{\Z}{1776\Z}$ 			\\ \hline 
	8 & $1 ,\alpha,\frac{\alpha^2 + 10\alpha}{25}$ 			& $40$ 	 & $\nicefrac{\Z}{2\Z}\times\nicefrac{\Z}{444\Z}$ 	\\ \hline 
	9  & $1 ,\frac{\alpha}{2},\frac{\alpha^2}{20}$ 			& $25$ 	& $\nicefrac{\Z}{592\Z}$\\ \hline
	10 & $1 ,\alpha,\frac{\alpha^2 + 10\alpha}{50}$ 			& $20$ 	& $\nicefrac{\Z}{444\Z}$ \\ \hline
	11 & $1 ,\alpha,\frac{\alpha^2 + 10\alpha}{100}$			& $10$ 	& $\nicefrac{\Z}{444\Z}$ \\ \hline    
	12 & $1 ,\frac{\alpha}{5},\frac{\alpha^2}{25}$			& $8$ 	& $\nicefrac{\Z}{2\Z}\times\nicefrac{\Z}{74\Z}$ \\ \hline    
	13 & $1 ,\frac{\alpha}{2},\frac{\alpha^2 + 10\alpha}{100}$		& $5$ 	& $\nicefrac{\Z}{148\Z}$ \\ \hline    
	14 & $1 ,\frac{\alpha}{5},\frac{\alpha^2}{50}$			& $4$ 	& $\nicefrac{\Z}{74\Z}$ \\ \hline    
	15 & $1 ,\frac{\alpha}{5},\frac{\alpha^2 + 10\alpha}{100}$		& $2$ 	& $\nicefrac{\Z}{74\Z}$ \\ \hline    
	16 & $1 ,\frac{\alpha}{10},\frac{\alpha^2}{100}$			& $1$ 	& $\nicefrac{\Z}{74\Z}$ \\ \hline    
    \end{tabular}
    }    
    \caption{The over-orders of $S_1$ from Example \ref{ex2}}
    \label{ex2:table}
  \end{table}
 
 Among the over-orders of $E$, the orders $S_2,S_4,S_6,S_7,S_9,S_{10},S_{14}$ are non-Gorenstein, so we also have the weak equivalence classes corresponding to $S_i^t$, for $i=2,4,6,7,9,10,14$.
 But unlike the previous example, there are two other weak equivalence classes, represented by the ideals $I=50\Z\oplus 10\alpha\Z \oplus 5\alpha^2\Z$ and $J=20\Z\oplus 10\alpha\Z\oplus 2\alpha^2\Z$, both with multiplicator ring $S_6$.
 This means that $\#\Wk(E)=25$ and using the information about the Picard groups of the over-orders we can deduce that $\#\ICM(E)=69116$.
\end{example}

\begin{example}
\label{ex:prodnotBass}
  Consider the irreducible polynomials $f_1=x^2 + 4x + 7$ and $f_2=x^3 - 9x^2 - 3x - 1$ and define $f=f_1f_2$.
  Put $K_1=\Q[x]/(f_1)$, $K_2=\Q[x]/(f_2)$ and $K=\Q[x]/(f)\simeq K_1 \times K_2$.
  For $i=1,2$ denote by $R_i$ the monogenic order $\Z[x]/(f_i)$ and by $\cO_{i}$ the maximal order of $K_i$.
  Similarly, let $R$ be the monogenic order $\Z[x]/(f)$ and $\cO$ the maximal order of $K$.
  It is easy to verify that $[\cO_{i}:R_i]=2$ for both $i=1$ and $i=2$. 
  Therefore, for both $i=1$ and $i=2$, the only over-order of $R_i$ is $\cO_{i}$ and hence $R_i$ is a Bass order.
  In particular, it follows that
  \[ \ICM(R_i) = \Pic(R_i) \sqcup \Pic(\cO_{i}). \]
  We can check that $\Pic(R_1)$ is trivial and hence $\cO_{1}$ is a principal ideal domain, since the extension map
  $I\mapsto I\cO_{1}$ induces a surjective group homomorphism from $\Pic(R_1)$ to $\Pic(\cO_{1})$, see Remark \ref{rmk:computePic}.
  We deduce that
  \[ \ICM(R_1) = \set{ \isoclass{R_1},\isoclass{\cO_{1}} }. \]
  On the other hand, $\Pic(R_2)$ and $\Pic(\cO_{2})$ are both isomorphic to $\Z/3\Z$ and generated respectively by $I=(69R_2+(28+\alpha_2+\alpha_2^2)R_2)$ and $J=I\cO_{2}$, where $\alpha_2 = x \bmod f_2$.
  It follows that
  \[ \ICM(R_2)=\set{ \isoclass{R_2},\isoclass{I},\isoclass{I^2},\isoclass{\cO_{2}},\isoclass{J},\isoclass{J^2}}. \]
  The situation for $R$ is much more complicated, as Figure \ref{ex3:fig} shows.
  \begin{figure}[h!]
  	\scalebox{.75}[.75]{
	\includegraphics[angle=90,origin=c,width=\textwidth]{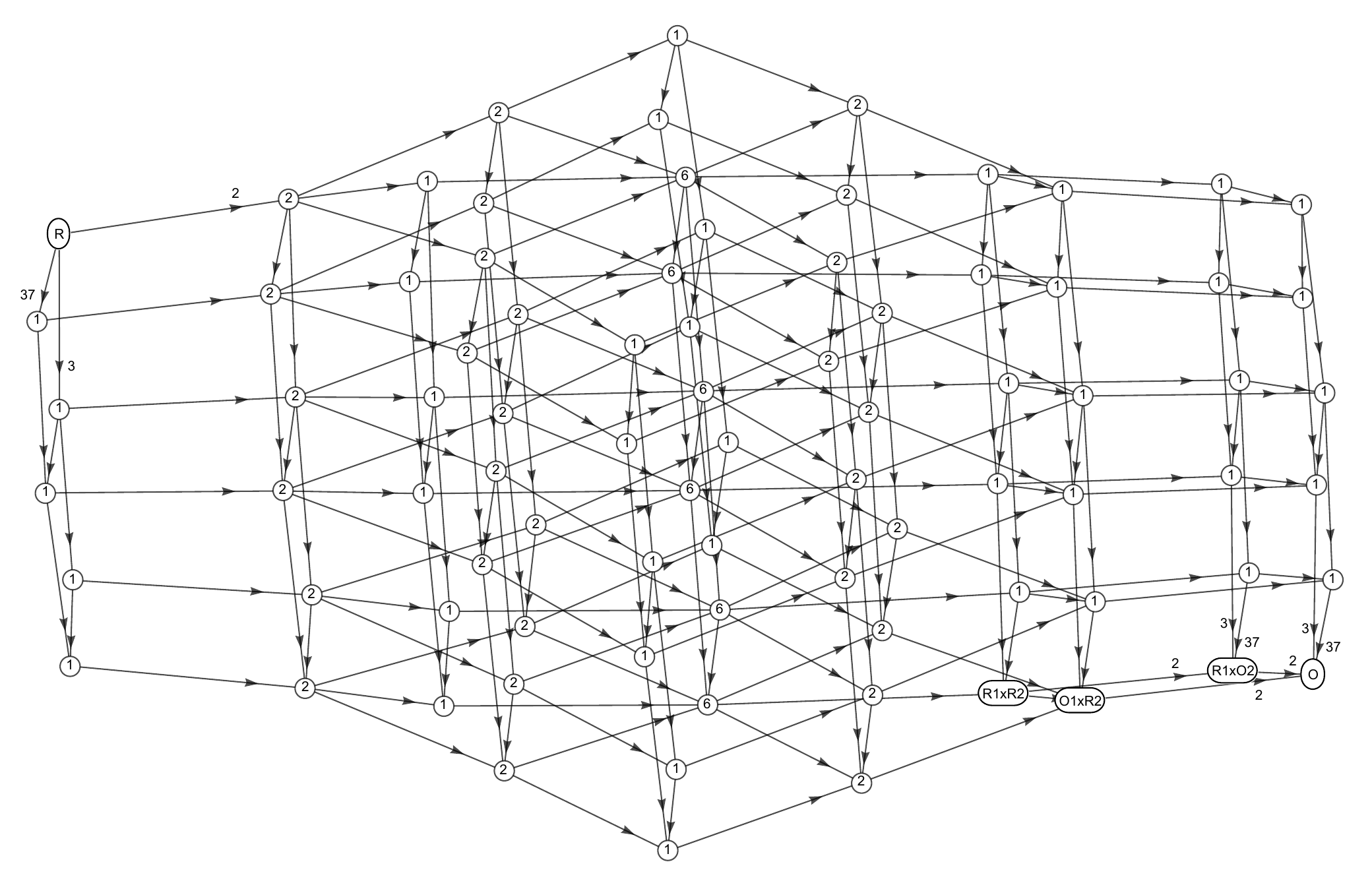}
	}
	\caption{The lattice of inclusions of the over-orders in Example \ref{ex:prodnotBass}. 
	The vertex corresponding to an order $S$ is labelled by the size of $\bWk{S}(R)$ except for the orders $R$, $R_1\times R_2$, $\cO_1\times R_2$, $R_1\times \cO_2$ and $\cO$ (which are all Gorenstein).
	The horizontal (resp.~vertical, diagonal) edges correspond to inclusions of index $2$ (resp~$3$, $37)$.
	We have labeled some of them at top-left and bottom-right corners for clarity.  }
    \label{ex3:fig}
  \end{figure}
  
  First of all, $R$ is not a product of orders of $K_1$ and $K_2$ and it can be computed that the index of $R$ in the maximal order $\cO\simeq\cO_{1} \times \cO_{2}$ is $21312=2^6 3^2 37$.
  We computed that $R$ has $84$ over-orders of which only $48$ are Gorenstein and hence $R$ is not Bass.
  We ran our algorithm for the ideal class monoid and we found that $\#\ICM(R)=852$.  
\end{example}


\section{Conjugacy classes of integral matrices}
\label{sec:matrixconj}


Recall that two $N\times N$ matrices $A$ and $B$ with entries in $\Z$ are \emph{conjugate} if there exists $O\in \GL_N(\Z)$ such that $OAO^{-1}=B$.
If this is the case, then $A$ and $B$ have the same minimal polynomial $m$ and the same characteristic polynomial $c$.
The converse is not true in general.
We will write $[A]_{\sim_\Z}$ for the conjugacy class of the matrix $A$.
In what follows we will describe how to compute representatives of the conjugacy classes when the minimal polynomial is square-free.
Our result is a generalization of \cite{LaClMD33}, where the authors treat the case $m=c$, which was then re-proved with a different method in \cite{taussky49}, with the extra assumption that $m=c$ is irreducible.
Note that Theorem \ref{thm:matrixconj} has independently been proved in \cite[Theorem 1.4]{HusertPhDThesis17} in greater generality.

Let $f_1,\ldots,f_r$ be a collection of distinct irreducible monic polynomials with integer coefficients and let $e_1,\ldots,e_r$ be positive integers such that $m=\prod f_i$, $c=\prod f_i^{e_i}$.
Put $N=\deg(c)$ and denote by $\cM_{m,c}(\Z)$ the set of integral $N\times N$ matrices with minimal and characteristic polynomials $m$ and $c$, respectively.

For every $i=1,\ldots,r$ put $K_i=\Q[x]/(f_i)$ and let $\Delta_i$ be the diagonal embedding of $K_i$ into $K_i^{e_i}$.
Define $\Delta$ as the product map $\prod_i\Delta_i$ with codomain $K=\prod_i K_i^{e_i}$. 
Observe that the order $R_0=\Z[x]/(m)$ has total quotient ring the $\Q$-algebra $\prod_i K_i$.
Denote with $R$ the image of $R_0$ in $K$ via $\Delta$ and put $\alpha=\Delta(x \mod (m))$. Let $\cL(R,K)$ be the set of full lattices in $K$ which are $R$-modules and pick $I$ in $\cL(R,K)$.
Fix a $\Z$-basis $\bar w=\set{w_1,\ldots ,w_n}$ of $I$.
The $R$-linear endomorphism of $I$ given by multiplication by $\alpha$ can be represented with respect to $\bar w$ by an integral matrix $A=A(I,\bar w)$ which lies in $\cM_{m,c}(\Z)$.
Clearly this representation depends on the choice of the $\Z$-basis of $I$.
If we change the $\Z$-basis of $I$ by a matrix $O\in \GL_N(\Z)$ then the multiplication by $\alpha$ will be represented by $O^{-1}AO$.
Hence we have a well defined map $I\to [A]_{\sim_\Z}$.
\begin{thm}
\label{thm:matrixconj}
 The association $\Phi:I\mapsto [A(I,\bar w)]_{\sim_\Z}$ induces a bijection
 \begin{align*}
 \tilde\Phi: \faktor{\cL(R,K)}{\simeq_R} & \longrightarrow \faktor{\cM_{m,c}(\Z)}{\sim_\Z},\\
	      \isoclass{I}  		 & \longmapsto [A(I,\bar w)]_{\sim_\Z}
 \end{align*}
 where $\simeq_R$ denotes isomorphisms of $R$-modules.
\end{thm}
\begin{proof}
 First we prove that the map $\tilde\Phi$ is well defined, that is that if $\vphi:I\rightarrow J$ is an $R$-linear isomorphism then $\Phi(I)=\Phi(J)$.
 Let $\bar w$ be a $\Z$-basis of $I$ and $\vphi(\bar w)$ the induced $\Z$-basis of $J$.
 Since $\vphi$ is $R$-linear, we have that $A(I,\bar w) = A(J,\vphi(\bar w))$, which implies that $\Phi(I)=\Phi(J)$.
 
 We now prove that $\tilde\Phi$ is injective. Let $I$ and $J$ be in $\cL(R,K)$ and fix a $\Z$-basis, say
 \[ I=w_1\Z\oplus\ldots\oplus w_N\Z \]
 and
 \[ J=v_1\Z\oplus\ldots\oplus v_N\Z. \]
 Assume that $\Phi(I)=\Phi(J)$, that is $A(I,\bar w)= O^{-1} A(J,\bar v)O $ for some $O\in \GL_N(\Z)$.
 By acting with $O^{-1}$ 
 on $\bar v$ we find a new $\Z$-basis $\bar v'$ for $J$ such that $A(I,\bar w)= A(J,\bar v')$.
 Now the $\Z$-linear bijection $I\to J$ defined by $w_i\mapsto v'_i$ commutes with multiplication by $\alpha$, since the matrices representing the operation with respect to $\bar w$ and $\bar v'$ are the same, and hence it is an $R$-linear isomorphism.
 Therefore $\isoclass{I}=\isoclass{J}$ and $\tilde \Phi$ is injective.
 
 To conclude we need to prove that $\tilde \Phi$ is also surjective.
 We will do this by explicitly producing a map
 \[ \Psi:\cM_{m,c}(\Z) \to \cL(R,K)/\simeq_R\]
 which descends to a section $\tilde \Psi$ of $\tilde\Phi$.
 Let $A$ be a matrix in $\cM_{m,c}(\Z)$.
 Note that since $m$ is square-free, $A$ is semisimple.
 Denote the element $(x \bmod f_i)$ of $K_i$ by $\alpha_i$.
 Note that
 \[ \alpha=(\underbrace{\alpha_1,\ldots,\alpha_1}_{\text{$e_1$ times}},\ldots,\underbrace{\alpha_r,\ldots,\alpha_r}_{\text{$e_r$ times}} ). \]
 Let
 \begin{equation}
 \label{eq:eigenvectors}
 v_{i,1},\ldots,v_{i,e_i} \in K_i^N 
 \end{equation}
 be a basis of the eigenspace corresponding to $\alpha_i$,
 that is, linearly independent vectors such that
 \[A v_{i,j_i}=\alpha_iv_{i,j_i}  \]
 for each $i=1,\ldots, r$ and $j_i=1,\ldots, e_i$.
 Let $E=e_1+\ldots +e_r$ and consider the $E \times N$ matrix whose rows are the vectors $v_{i,j_i}$ and denote by $w_k$ the $k$-th column, for $k=1,\ldots,N$.
 Observe that each $w_k$ is an element of $K$ and define
 \[I=\Span{ w_1,\ldots,w_N }_\Z \subset K.\]
 If $A=(a_{h,k})$ and $v_{i,j_i}=( v_{i,j_i}^{(1)},\ldots,v_{i,j_i}^{(N)} )$ then
 \[w_k=(v_{1,1}^{(k)},\ldots,v_{1,e_1}^{(k)},\ldots,v_{r,1}^{(k)},\ldots,v_{r,e_r}^{(k)})\]
 and it follows that
 \begin{align}
 \label{eq:multbyalpha}
 \begin{split}
  \alpha w_k & = ( \alpha_1 v_{1,1}^{(k)},\ldots,\alpha_1 v_{1,e_1}^{(k)},\ldots , \alpha_r v_{r,1}^{(k)},\ldots,\alpha_r v_{r,e_r}^{(k)} ) \\
	     & = ( \sum_{h=1}^N a_{k,h}v_{1,1}^{(h)},\ldots,\sum_{h=1}^N a_{k,h} v_{1,e_1}^{(h)},\ldots , \sum_{h=1}^N a_{k,h} v_{r,1}^{(h)},\ldots,\sum_{h=1}^N a_{k,h} v_{r,e_r}^{(h)} ) \\
	     & = \sum_{h=1}^N a_{k,h} (v_{1,1}^{(h)},\ldots,v_{1,e_1}^{(h)},\ldots,v_{r,1}^{(h)},\ldots,v_{r,e_r}^{(h)}) \\
	     & = \sum_{h=1}^N a_{k,h} w_h \in I,
 \end{split}     
 \end{align}     
 which implies that $I$ is closed under multiplication by $\alpha$, and hence it is an $R$-module.
 Moreover, \eqref{eq:multbyalpha} means that the multiplication by $\alpha$ is represented by the matrix $A$ with respect to the generators $w_1,\ldots,w_N$. 
 We prove now that $I$ is a full lattice, or equivalently that the $\Q$-vector space $V=I\otimes_\Z \Q$ equals $K$.
%
%
 Note that $A$ represents the $\Q$-linear map induced by multiplication by $\alpha$ on $V$.
 Since $A$ is semisimple there is a decomposition 
 \begin{equation}
  \label{eq:decsimplefactors}
  V=W_1\oplus \ldots \oplus W_r
 \end{equation}
 into $\Q$-vector spaces which are stable under the action of $\alpha$, and possibly after renumbering we can assume that $A|_{W_i}$ has minimal polynomial $f_i$ and hence that $W_i$ is a $K_i$-vector space.
 For each $i$, since the vectors $v_{i,1},\ldots,v_{i,e_i}$ are linearly independent over $K_i$, we see that $W_i$ must have dimension $e_i$.
 This concludes the proof that $I\in \cL(R,K)$.
 
 Observe that the construction of $I$ depends on the choice of eigenvectors in \eqref{eq:eigenvectors}.
 A different choice is attained by the action of a block-diagonal matrix $C$ in 
 \[
 \begin{pmatrix}
    \GL_{e_1}(K_1) & 0 & \ldots & 0 \\
    0 & \GL_{e_2}(K_2) & \ldots & 0 \\
    \vdots & & \ddots & \vdots\\
    0 & 0 & \ldots & \GL_{e_r}(K_r)
 \end{pmatrix}
 \]
 Observe that $C$ induces an $R$-linear automorphism of $K$ and hence its action on $I$ will also be an $R$-isomorphic object of $\cL(R,K)$.
 Hence we have a well defined map $\Psi$ which associates $A\mapsto \isoclass{I}$.
 
 If instead of $A$ we take a conjugate matrix $B$, it will reflect as taking an invertible $\Z$-linear combination of the eigenvectors in \eqref{eq:eigenvectors}.
 Clearly this will not change the $\Z$-span that they generate, that is, the lattice $I$, and hence $\Psi$ descents to a well defined map
 \[ \tilde\Psi: \faktor{\cM_{m,c}(\Z)}{\sim_\Z} \longrightarrow \faktor{\cL(R,K)}{\simeq_R} \]
 which by construction is a section of $\tilde \Phi$.
 This implies that $\tilde \Phi$ is surjective and concludes the proof.
\end{proof}
In general, the set of $R$-isomorphism classes in $\cL(R,K)$ is hard to handle, but under certain assumptions we can reduce it to an ideal class monoid computation.

\begin{cor}
\label{cor:computeconjclass}
 Let $f$ be a square-free monic integral polynomial and put $R=\Z[x]/(f)$.
 \begin{enumerate}[(a)]
  \item \label{cor:matrices:item1} There is a bijection
 \[
 \faktor{\cM_{f,f}(\Z)}{\sim_\Z}
 \longleftrightarrow
 \ICM(R).
 \]
 \item \label{cor:matrices:item2} Assume that $R$ is a Bass order. Let $N$ be a positive integer. We have a bijection
  \[
 \faktor{\cM_{f,f^N}(\Z)}{\sim_\Z}
 \longleftrightarrow
 \faktor{\cC}{\simeq_R},
 \]
 where the objects of $\cC$ are $R$-modules of the form $I_1\oplus\ldots \oplus I_N$, with $I_i$ fractional $R$-ideals satisfying $(I_i:I_i)\subseteq (I_{i+1}:I_{i+1})$, and two such modules $I_1\oplus\ldots \oplus I_N$ and $I'_1\oplus\ldots \oplus I'_N$ are isomorphic if and only if $(I_i:I_i)=(I'_i:I'_i)$ for every $i$ and $\isoclass{I_1\cdot \ldots\cdot I_N}=\isoclass{I'_1\cdot \ldots\cdot I'_N}$ in $\ICM(R)$.
 \end{enumerate}
\end{cor}
\begin{proof}
   Part \ref{cor:matrices:item1} follows from the equality $ \ICM(R) = \cL(R,R\otimes \Q)/\simeq_R $ proved in Theorem \ref{thm:matrixconj} and part \ref{cor:matrices:item2} is a direct consequence of the classification given in \cite[Theorem 7.1]{LevyWiegand85}.
\end{proof}
Observe that Corollary \ref{cor:computeconjclass} provides an algorithm to check whether two matrices, with characteristic polynomial of the appropriate form, are conjugate over $\Z$. 
This problem is solved in greater generality by a new algorithm proposed in \cite{EickHofmannOBrien19}.
In Example \ref{ex:timings} and Table \ref{table:timings} we compare the running time with our algorithm for matrices with square-free characteristic polynomial.

\begin{example}
   Let $f=f_1f_2$, $K$, $\alpha$ and $R$ be defined as in Example \ref{ex:prodnotBass}.
   Put $\alpha_1=x \mod f_1$ and $\alpha_2=x \mod f_2$ so that $\alpha = (\alpha_1,\alpha_2)$.
   Furthermore denote by $1_1$ and $1_2$ the images of the unit elements of $K_1$ and $K_2$ respectively under the canonical isomorphism 
   $K\simeq K_1\times K_2$.
   Define
   \[ \beta_1=(1_1,0),\quad\beta_2=(\alpha_1,0),\quad\beta_3=(0,1_2),\quad\beta_4=(0,\alpha_2),\quad\beta_5=(0,\alpha_2^2). \]
   Observe that
   \[\cA=\set{1,\alpha,\alpha^2,\alpha^3,\alpha^4}\]
   and
   \[\cB=\set{ \beta_1,\beta_2, \beta_3, \beta_4, \beta_5 }\]
   are two bases of $K$ over $\Q$. 
   Consider the ideal $I$ given by
   {\footnotesize \begin{equation*}
   \begin{split}
      I & = \dfrac{1}{444}\left(-13-46\alpha-138\alpha^2-50\alpha^3+7\alpha^4\right)\Z\oplus\dfrac{1}{222}\left(-9-29\alpha-87\alpha^2-9\alpha^3+2\alpha^4\right)\Z\oplus\\
        & \oplus\dfrac{1}{888}\left(883-12\alpha-36\alpha^2+32\alpha^3-3\alpha^4\right)\Z\oplus\dfrac{1}{888}\left(57+1084\alpha+588\alpha^2+168\alpha^3-25\alpha^4\right)\Z\oplus\\
        & \oplus\dfrac{1}{444}\left(190-99\alpha-75\alpha^2+5\alpha^3+3\alpha^4\right)\Z,
   \end{split}
   \end{equation*}}
   or equivalently
   {\footnotesize 
   \[ I = -2\beta_1\Z \oplus (\beta_1+\beta_2)\Z \oplus \frac12(5\beta_1+\beta_2+\beta_3)\Z 
   \oplus\frac12(5\beta_1+\beta_2+\beta_4)\Z \oplus \frac12(3\beta_1+\beta_2+\beta_3+\beta_5)\Z. \]
   }
   With respect to this $\Z$-basis of $I$, the multiplication by $\alpha$ is represented by the following integral matrix   
   \[A=\begin{pmatrix}
      -1 &  2 & 3 &  2 &  4 \\
      -2 & -3 & 0 &  0 & -4 \\
       0 &  0 & 0 & -1 & -4 \\
       0 &  0 & 1 &  0 &  2 \\
       0 &  0 & 0 &  2 &  9
   \end{pmatrix}.\]
   Performing an LLL reduction of the $\Z$-basis of $I$, we get
   {\footnotesize \begin{equation*}
      \begin{split}
         I & = \frac12(\beta_1-\beta_2+\beta_3+\beta_5)\Z\oplus\frac12(-\beta_1-\beta_2+2\beta_3)\Z
	   \oplus \frac12(-\beta_1+\beta_2+\beta_3+\beta_5)\Z \oplus \\ & \oplus\frac12(\beta_1+\beta_2+2\beta_3)\Z\oplus\frac12(-\beta_1-\beta_2+2\beta_4)\Z
      \end{split}
   \end{equation*}}
   and with respect to this $\Z$-basis of $I$ the multiplication by $\alpha$ is represented by
    \[A'=\begin{pmatrix}
	5 & 1 & 4 & -1 & 2 \\
	-6 & -3 & 0 & 2 & -3 \\
	4 & -1 & 5 & 1 & 0 \\
	2 & 3 & -4 & -2 & 2 \\
	2 & 1 & 2 & 1 & 0
    \end{pmatrix}.\]
    We find that for
    \[U=\begin{pmatrix}
	0 & 1 & 1 & 1 & 1\\
	-1 &-1 & 0 & 0 &-1\\
	0 & 1 & 0 & 1 & 0\\
	0 & 0 & 0 & 0 & 1\\
	1 & 0 & 1 & 0 & 0
    \end{pmatrix},\]
    we have
    \[A' =U^{-1}AU.\]
    We now follow the proof of Theorem \ref{thm:matrixconj} and we construct the ideal associated to the matrix $A$.
    The eigenvectors corresponding to the eigenvalues $\alpha_1$ and $\alpha_2$ are respectively
    \[ \left(1_1,\frac12(-\alpha_1 - 1_1),\frac14(-\alpha_1 - 5\cdot 1_1),\frac14(-\alpha_1 - 5\cdot 1_1),\frac14(-\alpha_1 - 3\cdot 1_1)\right)  \]
    and
    \[ \left(0,0,1,\alpha_2,\frac12(\alpha_2^2 + 1_2)\right). \]
    Hence we obtain
    \begin{align*}
       & w_1=\beta_1,\quad w_2=-\frac12(\beta_1+\beta_2),\quad w_3=\frac14(-5\beta_1 -\beta_2)+\beta_3,\\
       & w_4=\frac14(-5\beta_1 -\beta_2)+\beta_4,\quad w_5=\frac14(-3\beta_1 -\beta_2)+\frac12(\beta_3+\beta_5)
    \end{align*}
    and we put
    \[ J=\Span{w_1,w_2,w_3,w_4,w_5}_\Z. \]
    We find that $I$ and $J$ are isomorphic. More precisely, we have
    \[ (-2\beta_1+28\beta_4-3\beta_5)J = I. \]
\end{example}
\begin{example}
\label{ex:timings}
We compare the running time of our algorithm and the one in {\protect\cite{EickHofmannOBrien19}} to test conjugacy of integral matrices with square-free characteristic polynomial.
Table \ref{table:timings} contains the results of our computation on a Intel Xeon CPU E5-2697 v2 running at 2.70GHz.
The entry of the $k$-th row and $d$-th column consists of the pair $(t_1,t_2)$, where $t_1$ and $t_2$ are the average running times of our algorithm and the one in {\protect\cite{EickHofmannOBrien19}}, respectively, on $100$ pairs $(M_1,M_2)$ of $d\times d$ integer matrices built using the following proedure.
Pick a monic square-free polynomial $f\in \Z[x]$ with random coefficients with absolute value bounded by $k$ and let $C$ be companion matrix of $f$. 
Construct a random matrix $R_1$ in $\GL_d(\Z)$ by multiplying $10$ matrices of the form $I_d+E$, where $I_d$ is the identity in $\GL_d(\Z)$ and $E$ has exactly one non-zero entry, which is off the diagonal and has absolute value at most $k$.
Set $M_1=R_1^{-1}CR_1$.
Build $M_2$ in an analogous way. 
The characteristic polynomial of both $M_1$ and $M_2$ is $f$, so they correspond to fractional ideals of $\Z[x]/f$, see Corollary \ref{cor:computeconjclass}.\ref{cor:matrices:item1}.

\begin{table}[h]
\scalebox{.9}{
\begin{tabular}{cccccc}
\hline 
		  & $d=4$ & $d=6$ & $d=8$ & $d=10$ \\\hline
$k=20$  & $(0.69,1.2)$  & $(1.5,3.0)$   & $(3.2,9.6)$   & $(7.7,39.2)$  \\\hline
$k=40$  & $(0.75,1.2)$  & $(1.5,5.2)$   & $(3.6,16.3)$  & $(15.9,125.0)$
\\\hline
$k=80$  & $(0.69,1.6)$  & $(1.9,38.8)$  & $(6.9,80.9)$  & $(74.6,669.9)$
\\\hline
$k=160$ & $(0.69,5.2)$  & $(2.0,104.9)$ & $(26.2,294.2)$        & $(-,-)$
\\\hline
$k=320$ & $(0.75,33.2)$ & $(2.9,-)$     & $(53.9,-)$    & $(-,-)$       \\\hline
$k=640$ & $(0.81,127.7)$        & $(21.8,-)$    & $(238.2,-)$   & $(-,-)$
\\\hline
\end{tabular}
}	    
\caption{
Running-time comparison, see Example \ref{ex:timings}.
 }
\label{table:timings}
\end{table}
\end{example}
\begin{remark}
	Corollary \ref{cor:computeconjclass} allows us to produce representatives of the conjugacy classes of integral matrices with given characteristic polynomial (satisfying certain assumptions), solving \cite[Problem 7.7]{EickHofmannOBrien19} for such characteristic polynomials.
	Moreover, it gives an answer to the conjugacy problem over the integers, that is to determine whether two integral matrices $A$ and $B$ with square-free characteristic polynomial are conjugate.
   This was already considered in \cite{Gru80} where the author performs a series of reductions in order to translate the problem into an isomorphism test between fractional ideals of an integral domain.
   In this process the author has made a mistake, namely, the morphism $(3)$ on page $107$ is not a bijection, since the injective map from $R$ to the product of the monogenic orders is not surjective in general, which could lead to a very different output as Example \ref{ex:prodnotBass} shows.
\end{remark}

\bibliographystyle{amsalpha}
\renewcommand{\bibname}{References} 
\def\cprime{$'$}
\providecommand{\bysame}{\leavevmode\hbox to3em{\hrulefill}\thinspace}
\providecommand{\MR}{\relax\ifhmode\unskip\space\fi MR }
\providecommand{\MRhref}[2]{%
  \href{http://www.ams.org/mathscinet-getitem?mr=#1}{#2}
}
\providecommand{\href}[2]{#2}

\end{document}